\documentclass[12pt]{amsart}
\usepackage[dvips]{graphicx}
\usepackage{color}

 \usepackage{amsfonts, amssymb,verbatim,amsmath, latexsym, textcomp,amscd}

\def\CC{\mathbb C}

\def\EE{\mathbb E} 
 
\def\HH{\mathbb H}
\def\NN{\mathbb N}

\def\QQ{\mathbb Q}
\def\Qhat {\hat {\mathbb Q}}
\def\RR{\mathbb R}

\def\ZZ{\mathbb Z}

\def\bb{\beta}

\def\a{\alpha}

\def\d{\delta}
\def\e{\epsilon}
\def\g{\gamma}

\def\l{{\lambda}}
\def\m{{m}}
\def\q{q}

\def\s{\sigma}

\def\w{\eta}

\def\bu{{\bf u}}
\def\bv{{\bf v}}
\def\bw{{\bf w}}
\def\bx{{\bf x}}
\def\by{{\bf y}}
\def\bz{{\bf z}}

\def\cal{\mathcal}

\def\B{{ \mathcal B}}
\def\C{{ \mathcal C}}

\def\E{{\mathcal E}}
\def\F{{  \mathcal F}}

\def\N{{\cal N}}

\def\P{{\cal P}}

\def\S{{\mathcal S}}

\def\T{{\cal T}}

\def\V{{\mathcal V}}
\def\W{{\mathcal W}}

   \def\Ax{\mathop{\rm{Ax}}}

    \def\br{\mathop{\rm \bf{br}}}

   \def \inf{\mathop{\rm{inf}}} 
    \def\Int{\mathop{{Int}}}

  \def\SL{SL(2,\mathbb C)}

  \def\Tr{\mathop{\rm{Tr}}}

\def\dd{\partial}

\def\square{\hfill${\vcenter{\vbox{\hrule height.4pt \hbox{\vrule width.4pt
height7pt \kern7pt \vrule width.4pt} \hrule height.4pt}}}$}

\def\co{{\colon \thinspace}}

\newtheorem{theorem}{Theorem}[section]
\newtheorem{definition}[theorem]{Definition}
\newtheorem{lemma}[theorem]{Lemma}
\newtheorem{proposition}[theorem]{Proposition}
\newtheorem{corollary}[theorem]{Corollary}

\newtheorem{introthm}{Theorem}

\begin{document}
%%%%%%%%%%%%%%%%%%%%%%%%%%%%%%%%%%%%%%%%%%%%%%%%%%%%%%%%%%%%%%%%%%%%%
\title{Primitive stability and Bowditch's BQ-condition are equivalent}

 \author{Caroline Series}

\address{\begin{flushleft} \rm {\texttt{C.M.Series@warwick.ac.uk \\http://www.maths.warwick.ac.uk/$\sim$masbb/} }\\ Mathematics Institute, 
 University of Warwick \\
Coventry CV4 7AL, UK
\end{flushleft}}

 \begin{abstract}
We prove the equivalence of two conditions on the primitive elements in an $\SL$ representation of the free group  
$F_2 = <a,b>$, which may hold even when the $\SL$ image of $F_2$ is not discrete. One is Minsky's condition of primitive stability and the   other is the $BQ$-condition  introduced by Bowditch and generalised by Tan,  
Wong  and Zhang.   
\end{abstract}

\ 
 \date{\today}
\maketitle
 

 \section{Introduction}

 In this paper we show the equivalence of two conditions on the primitive elements in an $\SL$ representation $\rho$ of the free group  
$F_2 = <a,b>$ on two generators, which may hold even when the image $ \rho(F_2)$ is not discrete. One is the condition of primitive stability $PS$ introduced by Minsky~\cite{Minsky} and the other is the so-called $BQ$-condition  introduced by Bowditch~\cite{bow_mar} and generalised by Tan,  Wong  and Zhang~\cite{tan_gen}.
The same result has been proved independently by  Jaejeong Lee and Binbin Xu~\cite{binbin}.

To facilitate the proof we introduce a third condition which we call the \emph{bounded intersection property} $BIP$ which as we will show is implied by but does not imply either of the other two.

We begin by explaining these three conditions one by one.
Recall that an element $ u \in F_2$ is called \emph{primitive} if it forms one of a generating pair $(u,v)$ for $F_2$. Let $\P$ denote the set of primitive elements in $F_2$.  It is well known that up to inverse and  conjugacy, the primitive elements are enumerated by the rational numbers $\Qhat = \QQ \cup \infty$, see Section~\ref{farey} for details.

\subsection{The primitive stable condition $PS$}

The notion of primitive stability was introduced by Minsky in~\cite{Minsky} in order to construct an $Out(F_2)$-invariant subset of the $\SL$ character variety $\chi(F_2)$ strictly larger than the set of discrete free representations.

Let $d(P,Q)$ denote the  hyperbolic distance  between points $P, Q$
in hyperbolic $3$-space $\HH^3$. 
Recall that a path $t \mapsto \gamma(t) \subset \HH^3$  for $t \in I$ (where $I$ is a possibly infinite interval in $ \RR$) is called a $(K, \epsilon)$-quasigeodesic if there exist constants $K, \epsilon >0 $ such that 
\begin{equation} \label{qgeod} K^{-1}|s-t| - \epsilon \leq d(\gamma(s), \gamma(t)) \leq K|s-t| + \epsilon  \ \ \mbox{\rm {for all}} \ \  s, t \in I. \end{equation}

For  a  representation $\rho \co F_2 \to \SL$, in general we will denote elements in $ F_2$ by  lower case letters and their images under $\rho$ by the corresponding upper case, thus $X = \rho (x)$ for $x \in F_2$. Thus if $(u,v)$ is a generating pair  for  $F_2$  we write $U = \rho(u), V = \rho(v)$.  

Fix once and for all a base point $O \in \HH^3$ and suppose that $w= e_{1} \ldots e_{n} , e_{k} \in  \{u^{\pm}, v^{\pm} \} , i = 1, \ldots, n$ is a cyclically shortest word in the generators $(u,v)$.
The \emph{broken geodesic} $\br _{\rho}(w; (u,v))$  of $w$ with respect to  $(u,v)$  is the infinite path of geodesic segments joining vertices $$\ldots,  \ E_{n-1}^{-1}  E_n^{-1}O, E_n^{-1}O, O, E_1O, E_1 E_2 O,  \ldots,  E_1E_2 \ldots E_n  O, E_1E_2 \ldots E_n E_1 O, \ldots. $$
where $E_i = \rho(e_i)$.   
\begin{definition} \label{definePS}
Let $(u,v)$ be a fixed generating pair for $F_2$. A representation $\rho \co F_2 \to \SL$ is \emph{primitive stable}, denoted $PS$,  if the broken geodesics $\br_{\rho} (w; (u,v))$  for all cyclically shortest words $w= e_{1} \ldots e_{n}  \in \P, e_{k} \in \{ u^{\pm},v^{\pm}\}, k = 1 \ldots, n$,  are  uniformly $(K,\e)$-quasigeodesic for some fixed constants $(K,\e)$.\end{definition}
Notice that this definition is independent of the choice of basepoint $O$ and makes sense since  the change
from 
$\br_{\rho} (w; (u,v))$ to $\br_{\rho} (w; (u',v'))$ for some other generator pair $(u',v')$ 
changes all the constants for all the quasigeodesics uniformly.   

For   $g \in F_2$ write $||g||$  or more precisely $||g||_{u,v}$ for the word length of $g$, that is the shortest representation of $g$ as a product  of generators $(u,v)$.
It is easy to see that for fixed generators, the condition  $PS$ is equivalent to the existence of $K, \epsilon >0 $ such that 
\begin{equation} \label{qgeod1}
K^{-1}||w|| - \epsilon \leq d(O, \rho(w)O) \leq K||w|| + \epsilon \end{equation}
for all  finite cyclically shortest words $ w$  which are subwords of the infinite reduced word    
$ \ldots e_{1} \ldots e_{ n} \ldots e_{ 1} \ldots e_{ n}\ldots $

Recall that an  irreducible representation $\rho \co F_2 \to \SL$  is determined up to conjugation by the traces of $U = \rho(u),V = \rho(v)$ and $UV= \rho(uv)$ where $(u,v)$ is a generator pair for $F_2$. More generally, if we take the GIT quotient of  all (not necessarily irreducible) representations, then the resulting $SL(2,\CC)$ character variety of $F_2$ can be identified with $\CC^3$ via these traces, see for example \cite{goldman2} and the references therein.   (The only non-elementary (hence reducible) representation occurs when
$\Tr [U,V] = 2$. We  exclude this from the discussion, see for example~\cite{sty} Remark 2.1.)

\begin{proposition}[\cite{Minsky} Lemma 3.2] \label{prop:psopen} The set of primitive stable $\rho \co F_2 \to \SL$ is open  in the $\SL$ character  variety of $F_2$.  \end{proposition}

 Minsky showed that not all $PS$ representations are discrete.

\subsection{The Bowditch $BQ$-condition}

The notion of primitive stability was introduced by Bowditch in~\cite{bow_mar} in order to give a purely combinatorial proof of McShane's identity.

Again let $(u,v)$ be a generator pair for $F_2$ and  let $\rho \co F_2 \to \SL$. 

\begin{definition} \label{defineBQ}
Following~\cite{tan_gen}, an irreducible  representation $\rho \co F_2 \to \SL$ is said to satisfy the $BQ$-condition if 
\begin{equation}  \label{eqn:B2}
 \begin{split}  & \Tr \rho(g) \notin [-2,2]  \ \ \forall g \in \P    \ \ \mbox {\rm and}   \ \ \cr
& \{ g \in \P: |\Tr \rho(g)| \leq 2 \} \ \mbox {\rm is finite}.\end{split}\end{equation}
\end{definition}

We denote the set of all representations satisfying the $BQ$-condition by $\B$.

\begin{proposition}[\cite{bow_mar} Theorem 3.16, \cite{tan_gen} Theorem 3.2] \label{prop:bqopen}  The set    $\B$ is open in the $\SL$ character  variety of $F_2$.    \end{proposition}

Bowditch's original work~\cite{bow_mar} was on the case in which  the commutator $[X,Y] = XYX^{-1}Y^{-1}$
is parabolic and $\Tr  [X,Y] =  -2 $.  He conjectured that all  representations in $\B$ of this type are quasifuchsian and hence discrete. While this question remains open, it is shown  in~\cite{sty} that without this restriction, there are definitely  representations in $\B$ which are not discrete.

\subsection{The bounded intersection property BIP}\label{sec:BIP}

Recall that a word $w= e_1e_2 \ldots e_n$ in generators $(u,v)$  of $F_2$ is \emph{palindromic}  if  it reads the same forwards and backwards, that is, if $  e_1e_2 \ldots e_n = e_ne_{n-1} \ldots e_1$. Palindromic words have been studied by 
Gilman and Keen in~\cite{gilmankeen1, gilmankeen2}.

Suppose that $\rho \co F_2 \to \SL$ and let $(u,v)$ be a generating pair. Denote  the extended common perpendicular of the axes of $U = \rho(u), V = \rho(v)$ by   $\E(U,V)$. 
By applying the $\pi$ rotation about $\E(U,V)$, it is not hard to see  that a word $w$ is palindromic in a generator pair $(u,v)$  if and only if  the axis of $W = \rho(w)$ intersects   $\E(U,V)$ perpendicularly, see for example~\cite{BSeries}. 

Fix  generators $(a,b)$ for $F_2$. We call the pairs $(a,b), (a,ab)$ and $(b, ab)$ the  \emph{basic generator pairs}.  (The order $ab$ or $ba$ is fixed but not important, see below.)
Now given $\rho \co F_2 \to \SL$ let $A = \rho(a), B = \rho(B)$ and consider the  three common perpendiculars $\E(A,B), \E(A,AB)$ and $\E(B,AB)$. 
 (We could equally well  chose  to use $BA$ in place of $AB$; the main point is that the choice is fixed once and for all.) 
We call  these lines the \emph{special hyperelliptic axes}.

\begin{definition} \label{defineBIP} 
Fix a basepoint $O \in \HH^3$. A representation 
$\rho \co F_2 \to \SL$ satisfies the \emph{bounded intersection property} $BIP$  if there exists  $D>0$ so that  
if a generator $w$ is palindromic with respect to one of the three basic generators pairs, then its axis intersects the corresponding special hyperelliptic axis in a point at distance at most $D$ from $O$. Equivalently, the axes of all palindromic primitive elements intersect the appropriate hyperelliptic axes in bounded intervals.
\end{definition}
Clearly this definition is independent of the choices of $(a,b)$ and $O$. 

A similar condition but related to \emph{all} palindromic axes was used in~\cite{gilmankeen2} to give a condition for discreteness of geometrically finite groups.

In Proposition~\ref{uniquepalindromes}, we show that every generator is conjugate to one which is palindromic with respect to one of the three basic generator pairs. In fact each primitive element can be conjugated (in different ways) to be palindromic with respect to two out of the three possible basic pairs, for a more precise statement see below.

\subsection{The main result} 

The main theorem of this paper is
\begin{introthm} \label{introthmA}The conditions $BQ$ and $PS$ are equivalent. Both imply, but are not implied by, the condition $BIP$. \end{introthm}

 In the case of real representations, Damiano Lupi~\cite{lupi} showed by case by case analysis following \cite{goldman} that the conditions $BQ$ and $PS$ are equivalent.
 
 To see that $BIP$ does not imply the other conditions, first note  that conditions $PS$ and $BQ$ both imply that no element in $\rho(\P)$ is elliptic or parabolic. The condition $BIP$ rules out parabolicity (consider the fixed point of a palindromic parabolic element to be a degenerate  axis which clearly meets the relevant  hyperelliptic axis at infinity). However  the condition does not obviously rule out elliptic elements in $\rho(\P)$.
In particular, consider any  $SO(3)$ representation, discrete or otherwise. Here all axes are elliptic and all pass through a central fixed point which is also at the intersection of all three hyperelliptic axes.  Such a representation clearly satisfies $BIP$. 
  
\begin{comment}  
  It is not clear to us whether a real representation for which the images of a generator triple  all have traces in $(-2,2)$ corresponding to a sphere with three cone point satisfies $BIP$ or not. In this case the hyperelliptic axes form a triangle  joining  the fixed points of 
$A, B$ and $AB$. We have been unable to determine  whether or not every palindromic generator intersects the extended axes within bounded distance of this triangle or not. 

We have also not been able to determine either whether $BIP$ is open in the character variety, or produce examples of non-real non-elementary representations  which satisfy $BIP$ but not $PS$.

\end{comment} 

The plan of the paper is as follows. In Section~\ref{farey} we present background on the Farey tree and prove an important result on palindromic representation of primitive elements, Proposition~\ref{uniquepalindromes}. We also introduce Bowditch's condition of Fibonacci growth. In Section~\ref{Bowditchbackground}, we summarise Bowditch's method of assigning an orientation to  the edges of the Farey tree
which results in a characterisation of the BQ-condition in terms of the existence of a certain finite attracting subtree.

Using the condition of Fibonacci growth (see Definition~\ref{fibonaccidefn}), it is not hard to show that $PS$ implies $BQ$. This was proved  in~\cite{lupi} and  is sketched in Section~\ref{BQimpliesBIP}.  Developing the results of Section~\ref{Bowditchbackground}, we then make estimates which prove (Theorem~\ref{thm:BQimpliesBIP}) that $BQ$ implies $BIP$.

Finally in Section~\ref{quasigeodesics} we use Theorem~\ref{thm:BQimpliesBIP} to prove  that $BQ$ implies $PS$. 
The advantage of the condition $BIP$ over $BQ$  is that it ties down the location of  axes.
After some preliminary work on quasigeodesics,  which heavily relies on the condition $BIP$,  we obtain further results which will eventually allow us to control broken geodesic paths for all but finitely many generator pairs.  The proof of Theorem~\ref{introthmA} is completed by Theorem~\ref{prop:BIBQimpliesPS} which shows that $BQ$ implies $PS$.

We would like to thank Tan Ser Peow  and Yasushi Yamashita for initial discussions about this paper.  The work involved in Lupi's thesis~\cite{lupi}  also made a significant contribution. The idea  of introducing the condition $BIP$ arose while  trying to interpret  some very interesting computer graphics involving non-discrete groups made by Yamashita. We hope to return to this topic elsewhere.

 \section{Primitive elements, the Farey tree and Fibonacci growth}  \label{farey} The Farey diagram $\F$ as shown in Figures~\ref{fig:farey} and \ref{fig:colouredtree}  consists of  the images of the ideal triangle with vertices at $1/0,0/1$ and $1/1$ under the action of $SL(2,\ZZ)$ on the upper half plane, suitably conjugated to the position shown in the disk. The label $p/q$ in the disk is just the conjugated image of the actual point $p/q \in \RR$.

\begin{figure}[ht]
\includegraphics[width=5.5cm]{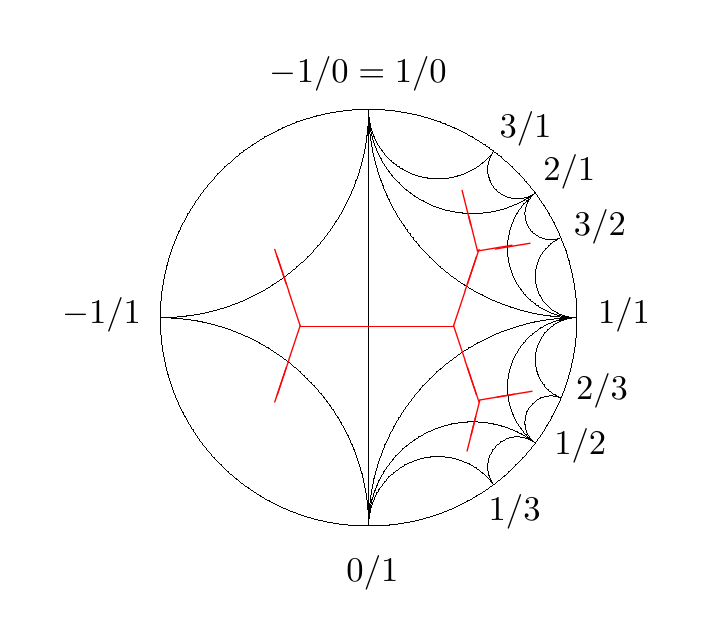}
\hspace{1cm}
\includegraphics[width=5.5cm]{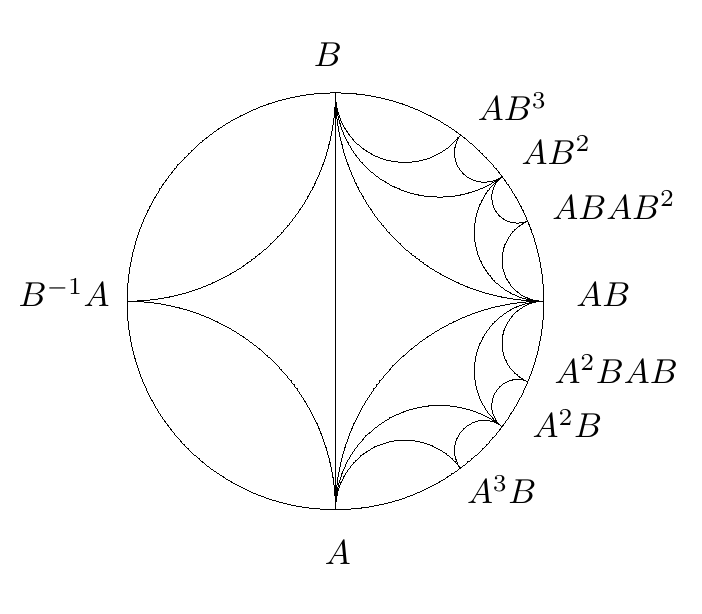}
\caption{The Farey diagram, showing the arrangement of rational numbers on the left with the corresponding  primitive words on the right.}\label{fig:farey}
\end{figure}

  Since the rational points in $\Qhat = \QQ\cup \infty$ are precisely the images of $\infty$ under $SL(2,\ZZ)$, they correspond bijectively to  the vertices of $\F$.
A pair  $p/q , r/s \in \hat \QQ$ are the endpoints of an edge  if and only if $pr-qs = \pm 1$; such pairs are called \emph{neighbours}.
A triple of points in $\hat \QQ$ are the vertices of a triangle precisely when they are the images of the vertices of the initial triangle $(1/0,0/1,1/1)$; such triples are always of the form 
$(p/q , r/s,(p+r )/( q+s))$ where  $p/q , r/s$ are neighbours.
In other words,  if $p/q , r/s$ are the endpoints of an edge, then the vertex of the triangle on the side away from the centre of the disk is found by `Farey addition' to be $(p+r )/( q+s)$. Starting from $1/0 = -1/0=  \infty$ and $0/1$, all   points in $\hat \QQ$ are obtained recursively  in this way. (Note we need to start with $-1/0=  \infty$  to get the negative fractions on the left side of the left hand diagram in Figure~\ref{fig:farey}.)

As noted in the introduction, up to inverse and conjugation, the primitive elements  in $F_2$ are enumerated by $\Qhat$.
Formally, we set $\overline \P $ to be the set of equivalence classes of primitive elements under the relation  $u \sim v$ if and only if either $v =  gug^{-1}$ or $v = gu^{-1}g^{-1}, g \in F_2$. We call the equivalence classes,  \emph{extended conjugacy classes}. In particular, the set of all cyclic permutations of a given word are in the same extended class. Since we are working in the free group, a word is  \emph{cyclically shortest} if it, together with all its cyclic permutations,  is reduced, that is, contains no occurrences of $x$ followed by $x^{-1}, x \in \{a^{\pm}, b^{\pm}\}$.  

The right hand picture in Figure~\ref{fig:farey}
 shows an enumeration of representative elements from $\overline \P $, starting with initial triple $(a,b,ab)$. Each vertex is labelled by a certain cyclically shortest generator $w_{p/q}$.  Corresponding to the process of Farey addition, the words  $w_{p/q}$ can be found by juxtaposition as indicated on the diagram.  Note that for this to work it is important to preserve the order:   if $u,v$ are the endpoints of an edge  with $u$ before $v$ in the anti-clockwise order round the circle, the correct  concatenation is $uv$. Note also that the words on the left side of the diagram involve $b^{-1}a$ corresponding to starting with $\infty = -1/0 $. 
It is not hard to see that pairs of primitive elements form a generating pair if and only if they are at the endpoints of an edge, while the words at the vertices of a triangle correspond to a generator triple of the form $(u,v,uv)$.

 The word $w_{p/q}$ is a representative of the extended conjugacy class identified with $p/q \in \hat \QQ $. We denote this class by $ [p/q]$.  
  It is easy to see that $e_a(w_{p/q}) / e_b(w_{p/q}) = p/q$, where $e_a(w_{p/q}),  e_b(w_{p/q})$ are the sum of the exponents in $w_{p/q}$ of $a,b$ respectively.  All other words in $ [p/q]$ are cyclic permutations of $w_{p/q}$ or its inverse. In what follows we largely focus on different representatives of the class which are palindromic with respect to one of the basic generator pairs, see Proposition~\ref{uniquepalindromes} below.

 \subsection{Generators and palindromicity}
Let $\EE = \{0/1,1/0, 1/1\}$ and define a map $\bb \co \hat \QQ \to \EE$  by $\bb(p/q) = \bar p /\bar q$, where   $\bar p ,\bar q$ are the  mod 2 representatives of $p,q$ in $\{0,1\}$. We refer to $\bb(p/q) $ as the mod 2  equivalence class of $p/q$. 
Say $p/q \in \Qhat $ is of type $\w \in \EE$ if  $\bb(p/q) = \w$. Say a generator $u \in F_2$ is of type $\w$ if $u \in [p/q]$ and $p/q$ is of type $\w$; likewise a generator pair $(u,v)$  is type $(\w, \w') $  if  $u, v$ are of types $\w, \w'$ respectively. 
As in Section~\ref{sec:BIP}, we fix once and for all a generator pair $(a,b)$ and identify $a$ with $0/1$, $b$ with $1/0 $ and $ab$ with $1/1$. The \emph{basic generator pairs} are the three (unordered) generator pairs $(a,b)$, $(a,ab)$ and $(b,ab)$ corresponding to $(0/1,1/0 )$, $(0/1,1/1 )$ and $(1/0,1/1 )$ respectively.  (Here the order $ba$ or $ab$ is not important but fixed.) 
For $\w,\w' \in \EE$ we say $u$ is palindromic with respect to $(\w,\w'),  \w \neq \w'$  if it is palindromic when rewritten in terms of the basic pair of generators corresponding to $(\w,\w')$; equally we say that a generator pair  $(u,v)$  is cyclically shortest (respectively palindromic with respect to the pair  $(\w,\w')$) if  each of $u,v$  have the same property. We refer to a generator pair  $(u,v)$ which is palindromic with respect to some pair of generators, as a \emph{palindromic pair}. Finally, say a generator pair $(u,v)$ is conjugate to a pair $(u',v')$ is there exists $g  \in F_2$ such that 
 $gug^{-1} = u'$ and $gvg^{-1} = v'$.

 \begin{proposition} \label{uniquepalindromes}
  If $u \in \P$  is of type $\w \in \EE$, then,  for each $\w' \neq \w$,  up to inverses there is exactly one conjugate generator $u'$ which  is cyclically shortest and palindromic with respect to $(\w,\w')$. 
If $(u,v)$ is a  generator pair of type $(\w,\w')$, then up to inverses, there is exactly one conjugate generator pair $(u',v')$ which  is cyclically shortest and palindromic with respect to $(\w,\w')$. 
  \end{proposition} 
\begin{proof} 
We begin by proving  the existence part of the second statement.  Observe that the edges of the Farey tree $\T$ may be divided into three classes, depending on the mod two equivalence classes of the generators labelling the neighbouring regions. 
In this way we may assign colours $r,g,b$ to the pairs 
$(0/1,1/0); (0/1, 1/1 ); (1/0, 1/1 )$ respectively and extend to a map $col$ from edges to $\{r,g,b\}$, see Figure~\ref{fig:colouredtree}. Note that no two edges of the same colour are adjacent, and that the colours round the boundary of each complementary region alternate.

\begin{figure}[ht]
\includegraphics[width=7.5cm]{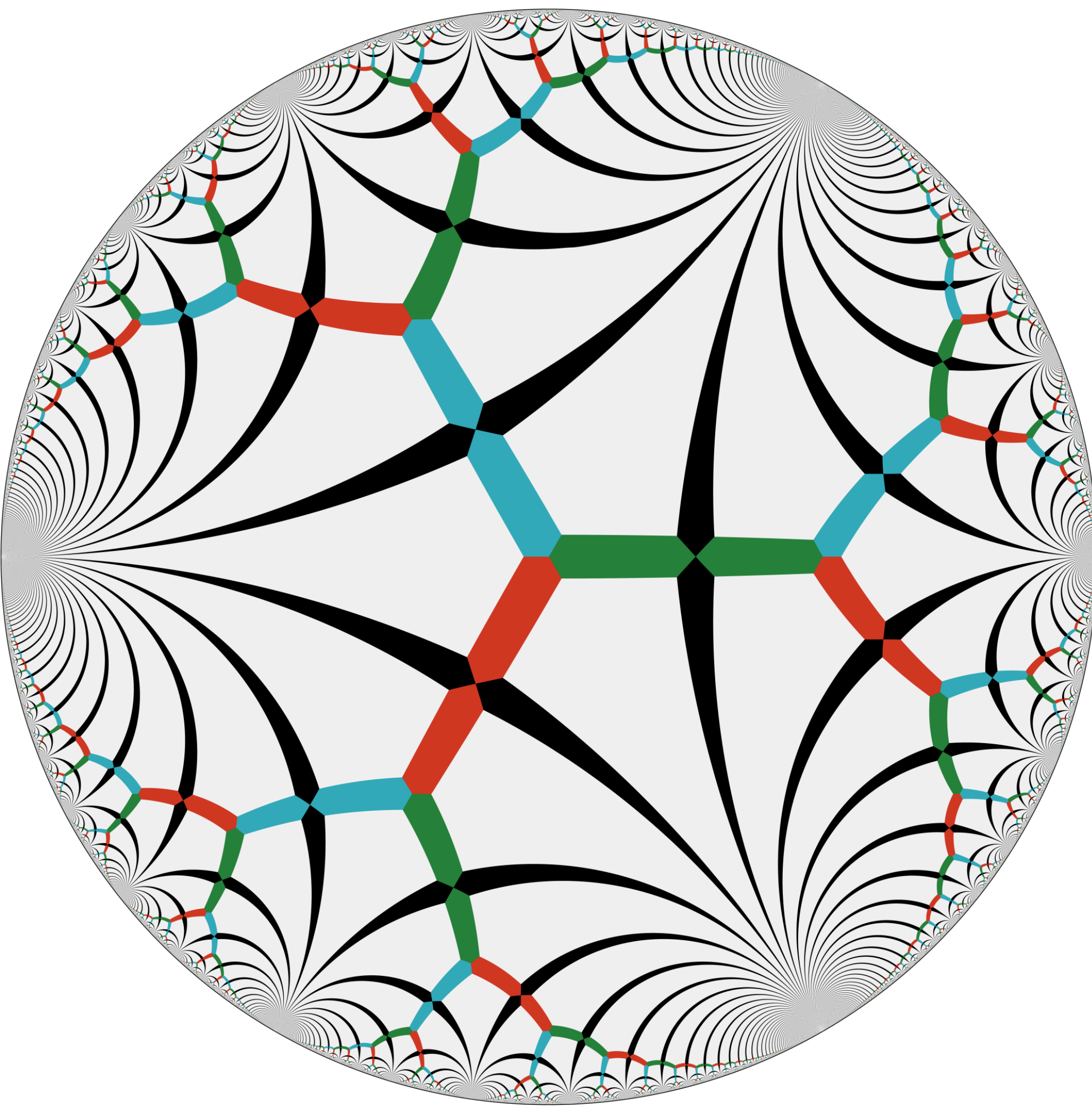}
\caption{The coloured Farey tree. The colours round the boundary of each complementary region alternate. The picture is a conjugated version of the one in Figure~\ref{fig:farey}, arranged so as to highlight the three-fold symmetry between $(a,b,ab)$. Image courtesy of Roice Nelson.}\label{fig:colouredtree}
\end{figure}

Let $e_0$ be the edge of $\T$ with adjacent regions labelled by $(a,b)$ and let $ q^+(e_0)$ and $ q^-(e_0)$ denote the vertices at the two ends of $e_0$, chosen so that the neighbouring regions are $(a,b, ab)$ and $(a,b, ab^{-1})$ respectively. 
Removing either of these two vertices disconnects $\T$. We deal first with 
the subtree  $\T^+$  consisting of the connected component of $\T \setminus \{q^-(e_0)\}$ which contains $q^+(e_0)$. Note that the regions adjacent to all edges of $\T^+$ correspond to non-negative fractions.  

Let $e$ be a given edge of $\T^+$ and let $q^+(e)$ denote the vertex of $e$ furthest from $q^-(e_0)$.    Let $\g = \g(e)$ be the  unique shortest edge path joining $q^+(e)$ to $q^-(e_0)$, hence including both $e$ and $e_0$. The \emph{coloured level} of $e$, denoted $col.lev(e)$,  is the number of edges $e'$  including $e$ itself in $\g(e)$  with $col(e') =col(e)$. 
Note that $\g(e)$ necessarily includes $e_0$, and, provided $ e \neq e_0$, one or other of the two edges emanating from $q^+(e_0)$ other than $e_0$.  Thus $col.lev(e) = 1$ for all three edges meeting  $q^+(e_0)$ while for all other edges of $\T^+$ we have $col.lev(e) > 1$.  

Now suppose that $e$ is the edge of $\T^+$ whose neighbouring regions are labelled by the given generator pair $(u,v)$.
The proof will be by induction on  $col.lev(e)$.

Suppose first $col.lev(e)=1$. If $e = e_0$ the result is clearly true, since the pair $(a,b)$  is palindromic with respect to itself.
The other two edges emanating from $q^+(e_0)$ have neighbouring regions corresponding to the base pairs $(a,ab)$ and $(ab,b)$, each of which pair is palindromic  with respect to itself, proving the claim.

Suppose the result is proved for all edges of coloured level $k \geq 1$. Let $e$ be an edge  whose adjacent generators are of type 
$(\eta, \eta')$. Suppose that $col(e) =c$ and let  $e'$ be the next  edge of $\g$   with $col(e') = c$ along the path $\g(e)$ from $q^+(e)$ to $q^-(e_0)$.  (Note that such $e'$ always exists  since $k+1\geq 2$.)
By the induction hypothesis there is a pair of generators  $(u,v)$ adjacent to $e'$ which is palindromic with respect the same basic generator pair  $(\eta, \eta')$.

Let $q^+(e')$ be the vertex of $e'$ closest to $e$, so that 
the subpath path $\g'$ of $ \g$ from $q^+(e')$ to $q^-(e)$ contains no other edges of  colour $c$, where $q^-(e)$ is the vertex of $e$ other than $q^+(e)$. Since there cannot be two adjacent edges of the same colour, the edges of $\g'$ must alternate between the two other colours.  This implies (see Figure~\ref{fig:colouredtree}) that $\g'$ forms part of the boundary of a complementary region $R$ of $\T^+$. Moreover the third edge at each vertex along $\dd R$ (that is, the one which is not contained in $\dd R$),  is coloured $c$.

Denote the generator associated to $R$ by $w$. Since the regions adjacent to  $ R$ and not in  $R$ include the one labelled $v$, the  labels of the regions around $\dd R$ can be written, in order,  in the form  $\ldots,w^{-2}v,  w^{-1}v, v, wv, w^2v, \ldots $.  Since $u$ is adjacent to $v$ in this list, 
either $ u = w^{-1}v$ or $ u = wv$. Since $e$ is coloured $c$ it points out of $\dd R$ so that the regions adjacent to $e$ also appear in this list so are of the form $(w^nv, w^{n+1}v)$ for some $n $.

Suppose $ u = w^{-1}v$.   Then  $(w^nv, w^{n+1}v) = ((vu^{-1})^nv, (vu^{-1})^{n+1}v)$. If $n \geq  0$ then this pair is clearly  palindromic with respect to $(u,v)$.
Since $(u,v)$ is palindromic with respect to $(\eta,\eta')$, it follows that so is $(w^nv, w^{n+1}v)$. If $n<0$ then noting that $(vu^{-1})^nv = (uv^{-1})^{-n}v$ we see that  $(w^nv, w^{n+1}v)$ is again palindromic in $(u,v)$ and hence with respect to $(\eta,\eta')$.
The argument in case $ u = wv$ is similar.

By the same  argument for the tree $ \T^-$ consisting of the connected component of $\T \setminus \{q^+(e_0)\}$ which contains $q^-(e_0)$ we arrive at the statement that the generators associated to each edge  of $ \T^-$ can be written in a form which is palindromic with respect to one of the three generator pairs associated to the edges emanating from $q^-(e_0)$, that is,
$(a, b^{-1})$, $(a, ab^{-1})$ or   $(ab^{-1}, b^{-1})$. The first pair is obviously palindromic with respect to $(a, b^{-1})$. 
Noting that   $ab^{-1} = a (b^{-1}a^{-1})a$ which is palindromic with respect to $(a,ab)$, the result follows.

Now we prove the existence part of the first claim. Suppose that $u \in \P$ is of type $\w \in \EE$ and that $\w' \neq \w$. Choose a generator $v$ of type $\w'$ so that $(u,v)$ is a generator pair. By the above there is a conjugate pair $(u',v')$ palindromic with respect to $(\w,\w')$ and $u'$ is a generator as required. 

To see that $u'$ is unique, 
suppose that cyclically shortest  primitive elements $u$ and $u'$ are in the same extended conjugacy class and are both palindromic with respect to the same pair of generators, which we may as well take to be $\{0/1,1/0\}$. Notice that $u$ necessarily has odd length, for otherwise the exponents of $a$ and $b$ are both even. 

Let $u = e_{r}  \ldots e_{1} f   e_{1} \ldots e_{r}$ and  suppose that $f' = e_{k}$ is the centre point about which $u'$ is palindromic  for some $1 \leq k \leq r$.   Then 
$\ldots uu \ldots $ is periodic with minimal period of length $2r+1$ and contains the subword $$  e_{r}  \ldots e_{1} f   e_{1} \ldots e_{{k-1}}f' e_{{k-1}} \ldots e_{1}  f e_{1}\ldots e_{r}$$    so after $f   e_{1} \ldots e_{{k-1}}f' e_{{k-1}} \ldots e_{1}$ the sequence repeats. Since this subword has length $2k<2r+1$ this contradiction proves the result.   

The claimed uniqueness of generator pairs follows immediately.
\end{proof}

 \subsection{Fibonacci growth}

Since all words in an extended conjugacy class have the same length, and since representative of the extended conjugacy class corresponding to $p/q \in \hat \QQ$ can found by concatenation starting from the initial generators $(a, b)$, it follows that $||w||_{(a,b)}= p+q$  for all $w \in [p/q]$.  
 This leads to the following definition from~\cite{bow_mar}:

 \begin{definition} \label{fibonaccidefn}   A representation $\rho \co F_2 \to \SL $ has \emph{Fibonacci growth}  if there exists  $c>0$ such that for all cyclically reduced words $w \in \P$ we have $  \log^+|\Tr \rho(w)| < c ||w||_{(a,b)}$ and moreover $  \log^+|\Tr \rho(w)| >  ||w||_{(a,b)}/c$
 for all but finitely many cyclically reduced $w \in \P$ where $\log^+x = \max \{0, \log |x| \}$. \end{definition}
Notice that although the definition is made relative to a fixed pair of generators for $F_2$, it is in fact independent of this choice.

The following  result is fundamental:
\begin{proposition}[\cite{bow_mar} Proof of Theorem 2, \cite{tan_gen} Theorem 3.3] \label{fibonacci} If $\rho \co F_2 \to \SL$ satisfies the $BQ$-condition  then $\rho$ has Fibonacci growth. \end{proposition}

\section{More on the Bowditch condition}\label{Bowditchbackground}

In this section we explain some further background to the BQ-condition. For more detail see~\cite{bow_mar} and~\cite{tan_gen}, and for a quick summary~\cite{sty}.
  As above, $\overline \P$ is identified $\Qhat$ and hence with the set $\Omega$ of complementary regions of the Farey tree $\T$.  We denote the region associated to a generator $u$ by $\bf u$, thus $\bf u' = \bf u$ for all $u' \sim u$.
 For a given representation $\rho \co F_2 \to \SL$,  note that $\Tr [U,V]$ and hence $\mu =  \Tr [A,B]+2$ is independent of the choice of generators of $F_2$, where as usual $U = \Tr \rho(u)$ and so on. 
 Since $  \Tr U$ is constant on extended equivalence classes of generators, for $\bu \in \Omega$ we can define $\phi(\bu) =  \phi_{\rho}(\bu)=  \Tr U$ for any $ u \in \bu$.  
 For notational convenience we will sometimes write $\hat \bu$   
 in place of $\phi(\bu)$.
 
 For matrices $X,Y \in \SL$  set $x = \Tr X, y = \Tr Y, z = \Tr XY$. Recall the trace relations:
\begin{equation}\label{eqn:inverse} \Tr XY^{-1} = xy-z \end{equation} and 
\begin{equation} \label{eqn:commreln}  x^2+y^2+z^2 = xyz + \Tr {[X,Y]} +2.
\end{equation}
Setting $\mu =  \Tr {[X,Y]} + 2$, this last equation takes the form 
$$x^2+y^2+z^2 - xyz = \mu.$$

As is well known and can be proven by applying the above trace relations inductively, if  $\bu,\bv,\bw$ is a triple of regions round a vertex  of $\T$, then $\hat \bu,\hat \bv,\hat \bw$ satisfy~\eqref{eqn:commreln}. (In particular, $ \Tr {[U,V]} $ is independent of  choice of generators.)
Likewise if $e$ is an edge of 
$\T$ with adjacent regions $\bu,\bv$ and if $\bw, \bz$ are the third regions at either end of $e$, then $\hat \bu,\hat \bv,\hat \bw, \hat \bz$ satisfy~\eqref{eqn:inverse}.  (A map  $\phi: \Omega \to \CC$ with this property is called a \emph{Markoff map}  in~\cite{bow_mar}.)

Given $\rho \co F_2 \to \SL$, define $\T_{\rho}$ to be the tree whose complementary regions are labelled by the function $\phi = \phi_{\rho}$.
Following Bowditch~\cite{bow_mar}, we orient the edges of $\T_{\rho}$ as follows. Suppose that labels of the regions adjacent to some edge $e$ are $\hat \bu,\hat \bv$ and the  labels of the two  remaining regions at the two end vertices are $\hat \bw,\hat \bz$ so that $\hat \bz = \hat \bu \hat \bv-\hat \bw$.
Orient $e$ by putting an arrow from $\hat \bz$ to $\hat \bw$ whenever $|\hat \bz| > |\hat \bw|$ and vice versa. If both moduli are equal,  make either choice; if the inequality is strict, say that the edge is \emph{oriented decisively}.

For any $m \geq 0$ and $\rho \co F_2 \to \SL$ define $\Omega_{\rho}(m) = \{ \bu \in \Omega | |\phi_{\rho}(\bu)| \leq m\}$.

Now we collect up some important results from~\cite{tan_gen} which generalise those of~\cite{bow_mar}. 

\begin{lemma}[{\cite[Lemma 3.7]{tan_gen}}]\label{forkvertex} %NEED
Suppose $\bu,\bv,\bw \in \Omega$ meet at a vertex $\q$  of $\T$ with the arrows on both the edges adjacent to $\bu$ pointing away from $\q$. Then either $|\phi(\bu)| \leq 2$ or $\phi(\bv) = \phi(\bw) = 0$.  In particular, if $\rho \in \B$ then $|\phi(\bu)| \leq 2$. \end{lemma}

\begin{lemma}[{\cite[Lemma 3.11]{tan_gen} and following comment}] \label{infiniteray}
Suppose $\beta$ is an infinite ray consisting of a sequence of edges of $\T_{\rho}$ all of whose arrows point away from the initial vertex. Then $\beta$ meets at least one region $\bu \in \Omega$  with $|\phi( \bu)| < 2$. Furthermore, if the ray does not follow the boundary of a single region, it meets infinitely many  regions with this property.
\end{lemma}

\begin{theorem}[{\cite[Theorem 3.1(2)]{tan_gen}}]\label{connected} %NEED
For any $m \ge 2$, the set  $\Omega_{\rho}(m)$  is connected. Moreover $|\Omega_{\rho}(m)|<\infty$ if and only if $\rho \in \B$. 
\end{theorem}
\begin{proof} The first statement is~\cite{tan_gen} Theorem 3.1(2). The statement on finiteness of $\Omega_{\rho}(m)$ follows  from Lemma~\ref{infiniteray} and finiteness of   $\Omega_{\rho}(2)$. \end{proof}

Let $\by_i, i \in \ZZ$ be the regions in order around the boundary $\dd \bu$ of a single region $\bu \in \Omega$. It is easy to see (see the proof of Proposition~\ref{uniquepalindromes})   that the values $\phi (\by_i)$ satisfy a simple  recurrence relation and hence   grow exponentially unless  $\phi(\bu)$ is in the exceptional set $E = [-2,2] \cup \{\pm \sqrt{\mu}\} \subset \CC$. If $\rho \in \B$ then 
by definition $\phi(\bu) \notin [-2,2]$, while if $\phi(\bu)= \pm \sqrt{\mu}$ the values approach zero in one direction round $\dd \bu$ (see \cite{tan_gen} Lemma 3.10) and hence $\rho \not\in \B$ since condition \eqref{eqn:B2} is not satisfied.
Thus we find: 

\begin{lemma}[{\cite[Lemma 3.20]{tan_gen}}] \label{finiteboundary} %NEED
Suppose that $\rho \in \B$ and $\bu \in \Omega$ and consider the regions $\by_i, i \in \ZZ$ adjacent to $\bu$ in order round $\dd \bu$. Then away from a finite subset, the values $|\phi_{\rho}(\by_i)|$  are increasing and approach  infinity as $ i \to \infty$ in both directions. Moreover there exists a finite segment of $\partial \bu$ such that the edges adjacent to $\bu$ and not in this segment are directed towards this segment.
 \end{lemma}

Let  $\vec e$  be a directed edge. Its \emph{head} and \emph{tail} are the two ends of $e$, chosen so that the arrow on $\vec e$ points towards its head.  Note that $\T \setminus \{\vec e\}$ has two components.
We define  the  \emph{wake} of $\vec e$, denoted $\W(\vec e)$, to be the set of regions whose boundaries are contained in the component of $\T \setminus \{\vec e\}$ which contains the tail of $\vec e$, together with the two regions adjacent to $\vec e$. (This is the subset of $\Omega$ denoted $\Omega^{0-}(\vec e)$ in~\cite{bow_mar} and \cite{tan_gen}.)

For $\bu \in \W(\vec e)$ let $d(\bu)$ be the number of edges in the shortest path from $\bu$ to the head of $\vec e$.
Following~\cite{tan_gen} P.777,  define a function $F_{\vec e}$ on $\W(\vec e)$ as follows:
$F_{\vec e}(\bw) = 1$ if $\bw$  is adjacent to $\vec e$ and 
$F_{\vec e}(\bu) = F_{\vec e}(\bv)+ F_{\vec e}(\bw)$  otherwise,  where  $\bv,\bw$ are the two regions meeting $\bu$ and closer to $\vec e$ than $\bu$, that is, with $d(\bv) < d(\bu),d(\bw) < d(\bu)$.   

We need the following refinement of Proposition~\ref{fibonacci}:  
\begin{lemma} \label{Increasing3} \label{fibonacciwake} Suppose that $\rho \in \B$
and that $\Vec e$ is a directed  edge such at most one of the adjacent regions is in $\Omega(2)$, and suppose that the arrows on the edges of $\W(\vec e)$ are all directed towards $\Vec e$. Then there exist $c>0, n_0 \in \NN$, independent of $\Vec e$ (but depending on $\rho$),  so that $\log |\phi_{\rho}(\bu)| \geq c F_{\vec e}(\bu)$ for all but at most $n_0$ regions  $\bu \in \W(\vec e)$. \end{lemma}
\begin{proof} This essentially Lemmas 3.17 and 3.19 of~\cite{tan_gen}, see also Corollary 3.6 of~\cite{bow_mar}.
We only need to see that the constants $c, n_0$ are independent of $\Vec e$. 
By Lemma 3.17 in~\cite{tan_gen}, 
if neither adjacent region to $\Vec e$ is in $\Omega(2)$, then it suffices to take $c = m - \log 2$ where $m = \min \{ \log 3, \inf \{ \log |\phi(\bu)|: \bu \notin \Omega(2) \}\}$ and $n_0 = 1$. Since the  sets  $ \Omega(3)$ and $\Omega(2)$ are finite for any $M$, and since either 
$ \Omega(3) \setminus \Omega(2) = \emptyset$ or  the infimum is a minimum, we have $ m - \log 2>0$  and the result follows. 

Equally, if one of the adjacent regions to $\Vec e$ is in $\Omega(2)$ then the constant $c$ in Lemma 3.19 and the number $n_0$ for which the inequality fails depends on the unique region $\bx_0 \in \Omega(2)$ adjacent to $\Vec e$. Since $\Omega(2)$ is finite once again these bounds are uniform independent of $\vec e$. \end{proof}

Finally, we will need the \emph{sink tree} defined in the course of the proof of Theorem 3.3 in~\cite{tan_gen} and explained in more detail in Theorem 2.7 of~\cite{sty}.
\begin{proposition}\label{sinktree} There is a finite connected non-empty subtree tree $T_F$  of $\T_{\rho}$ so that every  path of strictly decreasing arrows eventually lands on an edge of $T_F$.   
Moreover  $T_F$ contains all sink vertices and all edges abutting on any sink vertex. There is a constant $M_0 \geq 2$ so that  if regions $\bu,\bv$ are adjacent to an edge of $T_F$, then $|\Tr U|, |\Tr V| \leq M_0$. In particular, if $\bu$ is a region touching a sink vertex then $|\Tr U| \leq M_0$.
\end{proposition}
\begin{proof} Most of the assertions are proved  on p. 782 of ~\cite{tan_gen}, see also Corollary 3.12 of~\cite{bow_mar}. The assertion that $T_F$ contains all sink vertices is included in Theorem 2.7 of~\cite{sty}; this follows since $T_F$ is connected and the arrow on each edge not in $T_F$ points towards $T_F$.
Finally, to include all edges adjacent to any sink vertex we note that $T_F$ can always be enlarged, possibly increasing the constant $M_0$, to a larger finite tree with the same properties and which strictly contains the original one, see the proofs of Theorem 3.2 of~\cite{tan_gen} and Theorem 3.16 of~\cite{bow_mar}.
\end{proof}

\section{The Bowditch condition implies  Bounded Intersection}\label{BQimpliesBIP} 

In this section we prove some  implications among the three basic concepts. The first two results are easy:
 \begin{proposition} \label{prop:PSImpliesbounded} If a representation $\rho \co F_2 \to \SL$   is primitive stable then it satisfies $BIP$. 
 \end{proposition}
   \begin{proof} The broken geodesic corresponding to any primitive element by definition passes through the basepoint $O$.
   The  broken geodesics $\{\br_{\rho}(u ; (a,b)) \}, u \in \P$ are by definition uniformly quasigeodesic, so each is  at uniformly bounded distance to its corresponding axis. Hence all the axes are at uniformly bounded distance to $O$ and so in particular axes corresponding to primitive palindromic elements  cut the three corresponding special hyperelliptic axes  in  bounded intervals.\end{proof}

\begin{proposition}\label{PSimpliesBQ}  The condition $PS$  implies the Bowditch $BQ$-condition.
\end{proposition}
  \begin{proof} This is not hard, see for example~\cite{lupi}. From primitive stability,  uniformity of constants in \eqref{qgeod1}   implies Fibonacci growth, which in turn implies that only finitely many elements have lengths and therefore traces less than a give bound (see Lemma~\ref{compare2} below). \end{proof}

The main result of this section is:
\begin{theorem} \label{thm:BQimpliesBIP} The $BQ$-condition implies the bounded intersection property $BIP$.  
\end{theorem}
The idea of the proof is the following. Suppose that $(u,v)$ is a palindromic pair of generators, so that their axes intersect one of the three special  hyperelliptic axes $\E$ perpendicularly.  The hyperbolic cosine formula expresses the perpendicular distance $d$ between $\Ax U$ and $\Ax V$  (which is measured along $\E$) in terms of the translation lengths of $U, V$ and $U V^{-1}$. Provided these lengths are sufficiently long and that $ |\Tr U | \geq | \Tr U V^{-2}|$, we get an estimate showing  $d$ is exponentially small in the minimum of  $  \ell (U) $ and $\ell(V)$ (Proposition~\ref{shortdistances}). We then move stepwise along $\E$  from its intersection point of  with $\Ax U$  to its intersection point with 
one of the axes in  $\Omega(M)$ for suitable $M$ using intermediate intervals whose end points are the intersection points with $\E$ of axes corresponding to generator pairs,  all of which are palindromic with respect to the same basic generator pair as $(u,v)$ (Proposition~\ref{path}).  The estimates of Fibonacci growth as in Lemma~\ref{fibonacciwake} show that the sum of lengths of these intervals is finite proving the result. The details follow.

We begin with two easy results. 
 For a loxodromic element $X \in \SL$ let  $\ell(X)>0$ denote the (real) translation length and let  $\lambda (X) = (\ell(X) + i \theta(X))/2$ be \emph{half} the complex length,  so that $\Tr X = \pm 2 \cosh \lambda(X)$.
  \begin{lemma}\label{compare2} There exists  $L_0>0$ so that if 
 $\xi + i \eta \in \CC$ with $ \xi >L_0 $  then  
$\xi - \log 3 \leq \log |\cosh (\xi + i \eta)|  \leq \xi$. In particular, for $X \in \SL$ we have $e^{\ell(X)} /3\leq |\Tr X|/2 \leq e^{\ell(X)}$ whenever $\ell(X) > L_0$.  Also $ |\sinh (\xi + i \eta)| \geq  e^{\xi}  /3$.
\end{lemma}
\begin{proof}
For the right hand inequality, since $ |\cosh (\xi + i \eta)| = e^{\xi}   | (1+ e^{-2\xi -2i\eta} )|/2$ we have
$$\log |\cosh (\xi + i \eta)| = {\xi} + \log  | (1+ e^{-2\xi -2i\eta} ) |/2 \leq \xi$$  since  $| (1+ e^{-2\xi -2i\eta})  |/2 \leq 1$.

For the left hand inequality, since $\xi > L_0$ we have, choosing $L_0$ large enough, $| (1+ e^{-2\xi -2i\eta})  |/2 \geq 1/3$ so that $\log  | (1+ e^{-2\xi -2i\eta} ) |/2 \geq -\log 3$ and hence 
$\log |\cosh (\xi + i \eta)| \geq \xi -\log 3$.

The estimate on $ |\sinh (\xi + i \eta)|$ follows similarly.
\end{proof}

\begin{lemma} \label{compare3} Suppose that $ \rho \co F_2 \to \SL$ and that $(u,v)$ are generators  such that $|\Tr UV| \geq |\Tr UV^{-1}|$. If $  \ell(U) , \ell(V) >L_0$  with $L_0$ as in Lemma~\ref{compare2}, then    $   \ell(U)  + \ell(V) -2 \log 3  \leq   \ell(UV) $.
 \end{lemma}
\begin{proof}
Let $\hat u = \Tr U, \hat v = \Tr V, \hat z = \Tr UV, \hat w = \Tr UV^{-1}$, so that by assumption $|\hat z| \geq |\hat w|$.
Since $\hat u \hat v = \hat z+\hat w$ this gives $|\hat u \hat v| \leq  2|\hat z|$ and hence $\log |\hat u| + \log |\hat v|  \leq \log |\hat z|  + \log |2| $. Since $ \hat u = 2 \cosh (\l(U))$ and so on,  we have
\begin{equation}\label{compare1}
2 \log 2 + \log |\cosh \l(U)| + \log |\cosh \l(V)| \leq 2 \log 2 + \log |\cosh \l(UV)|.
\end{equation}
 Together with Lemma~\ref{compare2}  this gives
$$   \ell(U)  + \ell(V) -2 \log 3  \leq   \ell(UV) $$ as required.
 \end{proof}

 We start the proof of Theorem~\ref{thm:BQimpliesBIP} with an estimate of the perpendicular distance between the axes of  `long' pairs of generators. For $t \in \RR, f\co \RR \to \RR$ write $|f(t)| \leq O(t)$ to mean there exists $c>0$, depending only on the representation $\rho$,  such that $|f(t)| \leq ct$.

 \begin{proposition} \label{shortdistances} Suppose that $(u,v)$ is a pair of generators palindromic  with respect to one of the three basic generator pairs, and suppose that  $|\Tr U| \geq |\Tr UV^{-2}|$, where as usual $U = \rho(u), V = \rho(v)$. Then with $L_0>0$ as in Lemma~\ref{compare2},  if $\ell(U),  \ell(V) > L_0$ then  $d(\Ax U, \Ax V) \leq O(e^{ -\m })$ where $ m =   \min\{ \ell(U), \ell(V)\}$, with constants depending only on $(a,b)$ and $O$.  
  \end{proposition}
 \begin{proof}  
 Consider the right angled hexagon $H$ whose alternate sides follow the axes of $U, UV^{-1}$ and $V^{-1}$. Orienting the sides  consistently round $H$, we may define the complex distances $\s_1, \ldots , \s_6$ between the sides  in such a way that, assuming that $\Re \lambda(U) >0$ and so on,  we have  $\sigma_1 = \lambda(U)$ or $\lambda (U) + i \pi$; $\sigma_3 = \lambda(UV^{-1})$ or $\lambda (UV^{-1}) + i \pi$ and $\sigma_5 = \lambda(V^{-1})$ or $\lambda (V^{-1}) + i \pi$.   Moreover $\sigma_6 = \pm (d + i \theta)$ where $\d = d + i \theta$ is the complex distance between the correctly oriented axes of $U$ and $V$ and where we take $d \geq 0$. (See for example~\cite{serwolp} for a discussion of complex length and hyperbolic right angled hexagons, although as we shall see shortly such detail is not needed here.)

%\begin{figure}\label {fig:hex}
%\includegraphics{}
%\caption{The configuration of the right angled hexagon $H$}
%\end{figure}

 The hexagon formula in $H$ gives: 
 
\begin{equation}
  \cosh \s_6  = \frac{\cosh \s_3 - \cosh \s_1 \cosh \s_5    } {\sinh \s_1 \sinh \s_5}
 \end{equation}

Thus   \begin{equation}\label{hexformula}
|\cosh \s_6 +1| \leq \biggl{|}\frac{\cosh \s_3} {\sinh \s_1 \sinh \s_5}\biggr{|}+  | 1-\coth \s_1 \coth \s_5|
  \end{equation}
 which in view of the remarks above can be rewritten 
  \begin{equation}
|\cosh \d +1| \leq \biggl{|}\frac{\cosh \lambda(UV^{-1})} {\sinh \lambda(U) \sinh \lambda(V^{-1})}\biggr{|}+  |1- \coth \lambda(U) \coth \lambda(V^{-1})|
  \end{equation}

 Now assume that  $\ell(U),  \ell(V), \ell(UV^{-1})  > L_0$ with $L_0$ as in Lemma~\ref{compare2}.

For $z \in \CC$ we have $\coth z = (1 + e^{-2z})/(1 - e^{-2z})$, hence for large enough $|z|$ we have $\coth z  = 1 + O(e^{-2z})$ so that  
$$   \coth \lambda(U) \coth \lambda(V)   = 1 + O(e^{-2\m}), $$
where we use
 estimates as in Lemma~\ref{compare2}: $|\cosh (\xi + i \eta)|   \leq O(e^{\xi})$ and $|\sinh (\xi + i \eta)|   \geq O(e^{\xi})$ for  $\xi>L_0$.

Similar estimates also give 
\begin{equation}\label{hexformula1}
\biggl| \frac{\cosh \lambda(UV^{-1})  } {\sinh \lambda(U) \sinh \lambda(V)} \biggr| \leq c\exp(\ell(UV^{-1}) - \ell(U) - \ell(V))/2 .
   \end{equation}

By Lemma~\ref{compare3} applied to the generator pair $(V, UV^{-1})$, since by hypothesis $|\Tr U| \geq |\Tr UV^{-2}|$,
we have $\ell(U) \geq \ell(UV^{-1}) + \ell(V)  - 2\log3$ so that $-2\ell(V) \geq \ell(UV^{-1}) - \ell(V) -\ell(U) - 2\log3$
and hence $$\biggl| \frac{\cosh \lambda(UV^{-1})  } {\sinh \lambda(U) \sinh \lambda(V)}\biggr| \leq  O(e^{-\m }).$$

Now suppose we only have $\ell(U), \ell(V)  > L_0$ while  $\ell(UV^{-1}) \leq L_0$. Then instead of the estimate in~\eqref{hexformula1} we get \begin{equation}\label{hexformula2}
\biggl |\frac{\cosh \lambda(UV^{-1})  } {\sinh \lambda(U) \sinh \lambda(V)} \biggr | \leq O(e^{-\m })
   \end{equation} 
    with a bound independent of $(U, V)$. 
(In fact in this case the estimate improves to $O(\exp(- (\ell(U) + \ell(V)))$ but we won't need this here.)    

In either case we have $$| \cosh (\delta + i \pi) - 1| \leq  O(e^{-\m}) $$
and hence $d = \Re \delta = O(e^{-\m})$.  
\end{proof}

We now proceed to estimate the distance between arbitrary pairs of palindromic axes.  Assume that the representation $\rho \in \B$.	
Choose $M_0 \geq 2$ and a finite connected non-empty subtree tree $T_F$ of $\T$ as in Proposition~\ref{sinktree}.

 \begin{lemma} \label{plughole} Let $M \geq M_0$ and suppose that $\bu \in \Omega \setminus \Omega_{\rho}(M)$. Then there is an oriented edge $\Vec e$ pointing out of $\bu$ so that $ \Vec e$ is not contained in $ T_F$.  
\end{lemma}
   \begin{proof}
 Label the regions adjacent to $\bu$ consecutively round $\dd \bu$ by $\by_n, n \in \ZZ$ and let $e_n$ denote the edge between  $\by_n, \bu$. By  Lemma~\ref{finiteboundary}, for large enough $|n|$ the arrows on the edges round $\dd \bu$ point in the direction of decreasing $|n|$. Thus there is at least one  $r\in \ZZ$ 
so that the heads of $e_r$ and $e_{r+1}$ meet at a common vertex $\q \in \dd \bu$.  The remaining arrow at $q$ must point out of $\bu$ for otherwise $q$ is a sink vertex and  hence by  Proposition~\ref{sinktree} all the edges meeting at $q$ are  in $T_F$, so that $\bu \in \Omega(M_0) \subset \Omega(M)$ contrary to assumption.
\end{proof}

 We call such a vertex a \emph{plughole} of $\bu$.

  \begin{proposition} \label{path}
Let $M \geq M_0$ and suppose that the representation $\rho \in \B$ and that the generator $u$ is palindromic of type $\w \in \EE$. Suppose also that $\bu \notin \Omega(M)$. Pick $\w' \neq \w$. Then  there is a sequence of generators
$u_0 = u, u_1, \ldots u_k \in \P$ such that for $i = 0, \ldots, k-1$:
\begin{enumerate}
\item $(u_i, u_{i+1})$ are neighbours.
\item $(u_i, u_{i+1})$ are palindromic with respect to $(\eta, \eta')$.
\item $|\Tr U_i| \geq |\Tr U_{i}U_{i+1}^{-2}|$.
\item $\bu_k \in \Omega(M)$  but  $\bu_i \notin \Omega(M),   0 \leq i <k$.
\end{enumerate}
\end{proposition}
 \begin{proof}  
 
 Suppose that $ \bu \notin \Omega(M)) $ and let $\Vec e$ be an oriented edge pointing out of some plughole of $\bu$. 
 Of the two regions adjacent to $\Vec e$, one, $\bu'$ say, is of type $\w'$. Set $u_0 = u, u_1 = u'$ and arrange by cyclic permutation if necessary  that $(u_0, u_1)$ is palindromic with respect to $(\w, \w')$.
Then the other region adjacent to $\Vec e$ can be chosen to be $\bu \bf{u'^{-1}}$ and the region at the head of $\Vec e$ is $\bu \bf{u'^{-2}}$. Thus $|\Tr U| \geq |\Tr UU'^{-2}|$. 

If $ \bu_1 \in \Omega(M) $  conditions (1)-(4) are satisfied with $k=1$. Otherwise we repeat the argument. The process terminates because 
by Proposition~\ref{sinktree} every descending path of arrows eventually meets $T_F$, and both regions adjacent to an edge in $T_F$ are in $\Omega(M_0) \subset \Omega(M)$. 
\end{proof}

{\sc Proof of Theorem~\ref{thm:BQimpliesBIP}} 
 Suppose the generator $u = u_0$ is palindromic with respect $ \w$ and  that $\w' \neq \w$.  Let $\E$ be the corresponding special hyperelliptic axis.
 Choose $L> L_0$ as in  Lemma~\ref{compare2} so that   $|\Tr U| > 2 e^{L}$ implies $\ell(U) >L$. With $M_0$ as in Proposition~\ref{sinktree} choose $M = \max \{M_0, 2e^L\}$.
Let $\Xi$ denote the set of axes corresponding to elements in $\bv \in \Omega(M)$ which are of types either $\w$ or $\w'$. It is sufficient to see that   $\Ax U$ meets $\E$ at a uniformly bounded distance to one of the finitely many axes in $\Xi$.

Let $u_0 = u, u_1, \ldots u_k$ be the sequence of  Proposition~\ref{path}. If $k=0$ there is nothing to prove since $\Ax U_k \in \Xi$.

Suppose $k>0$. Let $\Vec e$ be the edge emanating from the plughole of $\bu_i, 0\leq i<k$ and consider the two adjacent regions $\bu_{i+1}, \bu_{i}\bu_{i+1}^{-1}$, choosing the numbering so that  $u_{i+1}$ is of type $\w'$, while $u_{i}u_{i+1}^{-1}$ is of the third type $\w''$. 
Since $\bu_i \notin \Omega(M)$  we have $\ell(U_i)>L$ by choice of $M$. If  in addition $\ell(U_{i+1})>L$ then the pair $(u_i, u_{i+1})$ satisfy the condition of Proposition~\ref{shortdistances} 
 so that $d(\Ax U_i, \Ax U_{i+1}) \leq O(e^{-m_i}), m_i =  \min \{\ell(U_i), \ell(U_{i+1}) \}ß$.
Otherwise, $\ell(U_{i+1}) \leq L$ so that $\bu_{i+1} \in \Omega(M) \subset \Xi$ and therefore  $k = i+1$ and the process terminates.

  Let $\Vec e$ be the oriented edge between $\bu_{k-1}, \bu_k$   and let  $\W(\Vec e)$ be its wake.  Then since the edge between 
  $\bu_i, \bu_{i+1}$ is always oriented towards $\Vec e$, we see that 
  $\bu_i \in \W(\Vec e), 0 \leq i \leq k$. 
  Let $F_{\vec {e}}$ be the 
 Fibonacci function   on $\W(\Vec e)$   defined  immediately above Lemma~\ref{fibonacciwake}. 
It is not hard to see that for $0 \leq i \leq k$ we have $F_{\vec {e}}(\bu_i) \geq k-i$.
By construction, $\bu_{k-1} \notin \Omega(M) $ so that  $\bu_{k-1} \notin \Omega(2)$.  Hence we can apply Lemma~\ref{fibonacciwake} to see that there exists $c>0, n_0 \in \NN$ depending only on $\rho$ and not on $\Vec e$ such that $\log^+{|\Tr U_i|} \geq c (k-i)$ for all but at most $n_0$   of  the regions $\bu_i$.

 Hence in all cases, for all except some uniformly bounded number of the regions $\bu_i$, $ \ell(U_i)   \geq c(k-i) -\log 2$ so that $m_i\geq c'(k-i) -c'$ for some fixed $c'>0$. Since all axes $\Ax U_i$  intersect  $\E$ orthogonally in points $P_i$ say, it follows that 
  $d(\Ax U_0, \Ax U_k) $ is the sum $\sum_0^{k-1} d(\Ax U_i, \Ax U_{i+1}) $ of the distances between the intersection points of $P_i, P_{i+1}$. Hence the distance from $\Ax U = \Ax U_0$  to one of the finitely many axes in $\Xi$ is uniformly bounded above, and we are done.
\qed

 \section{Bounds on broken geodesics} \label{quasigeodesics}

In this section we prove our main result Theorem~\ref{prop:BIBQimpliesPS}. We begin by collecting
some basic results on quasigeodesics.

By a \emph{broken geodesic} we mean a path composed of a sequence of geodesic segments \\ $\ldots, s_i, s_{i+1},  \ldots$ meeting at their endpoints $P_i, P_{i+1}$, where $P_i$ is the meeting of the end point of $s_i$ with the initial point of $s_{i+1}$. We call the $P_i$ \emph{bending points}  and define the \emph{exterior bending angle}  $\phi_i$ at $P_i$ to be the angle  between the extension of $s_i$ through $P_i$ and $s_{i+1}$. Thus $s_i,s_{i+1}$ combine to form a single longer geodesic segment iff $\phi_i = 0$. The \emph{interior bending angle}  is $\pi - \phi_i$.

The following three lemmas are well known, but for the reader's convenience we provide   proofs.

\begin{lemma} \label{brokengeod} 
Given any angle $\psi>0$, there exists $L = L(\psi)>0$, depending only on $\psi$, such that if $\gamma$ is any broken geodesic whose segments are of at length at least $L$, and such that the interior bending angle at each bending point is at least $\psi$, then $\gamma$ is a quasigeodesic with constants depending only on $\psi$.
\end{lemma}
\begin{proof} 
Let $L_1$ be length of the finite side of a triangle with angles $0,\pi/2,  \psi/2$. Suppose  that lines $QP, Q'P$ make an angle of  $\a >  \psi$ at $P$ and also that $|QP|, |Q'P| > L_1$. Let $\Lambda, \Lambda'$  be the lines through $Q,Q' $ and  orthogonal to $QP, Q'P$ respectively. Then $\Lambda, \Lambda'$ do not meet.

Now pick $L>L_1$ and consider a broken geodesic with segments $s_1, \ldots, s_n$  of lengths $\ell_1, \ell_2, \ldots, \ell_n$ with $\ell_i > 3L$   and such that the interior angles between segments $s_i, s_{i+1}$ are at least 
$\psi$ for all $i$. Let  $H_i^-, H_i^+$ be  half planes orthogonal to $s_i$ at distance $L$ from the initial and final points $s_i^-, s_i^+$ of $s_i$ respectively.  Clearly $H_i^- \cap  H_i^+ = \emptyset$ and $d(H_i^+ , H_{i}^-) \geq \ell_i - 2L$. Let  $\Pi$ be the plane containing $s_i$ and $ s_{i+1}$.
Then the lines $s_i, s_{i+1}$ together with the lines  $H_i^+ \cap  \Pi$, $H_{i+1}^- \cap \Pi$ are exactly in the configuration described in the first paragraph, and hence $H_i^+ \cap  H_{i+1}^- = \emptyset$.

This shows that the half planes $H_1^-, H_1^+, \ldots H_n^-, H_n^+$ are nested and that
$$d(s_1^-, s_n^+) \geq  \sum_{i=1}^n (\ell_i - 2L)> \sum_{i=1}^n \ell_i/3 $$ which proves the result.
\end{proof}

\begin{lemma}\label{Longbending} Suppose given a hyperbolic triangle $\Delta$ with side lengths $a,b,c$ opposite vertices $A,B,C$  and angle $\psi$ at vertex $C$. 
Given $k>0$ there exist  $L, \e >0$ such that if $c \geq a+b -k$ then $\psi >\e$ whenever $a,b >L$.
\end{lemma}
\begin{proof} The formula 
$$ \cosh c = \cosh a \cosh b - \sinh a \sinh b \cos \gamma$$
rearranges to 
$$ \cos \gamma -1 = \bigl(\frac{\cosh a \cosh b}{\sinh a \sinh b } - 1 \bigr) - \frac{\cosh c}{\sinh a \sinh b }.$$
Since  $$\frac{\cosh a \cosh b}{\sinh a \sinh b } \to 1 \ \ \mbox{\rm and} \ \   \frac{\cosh c}{\sinh a \sinh b } \geq \frac{4e^{ a+b-k}}{2 e^ a e^ b } \geq 2 e^{-k}$$
 as $a,b \to \infty$ 
 we see that $ \cos \gamma -1$ is bounded away from $0$ giving the required bound. 
\end{proof}

\begin{lemma} \label{singleqgeod} 
Let $w$ be a cyclically shortest word in $F_2$ and let $\rho \co F_2 \to \SL$. Suppose that the image   $W = \rho(w)$  is loxodromic. Suppose also that the generators $(u,v)$ have  
 images  $U = \rho(u),V = \rho(v)$. Then the broken geodesic $\br_{\rho}(w; (u,v))$ is quasigeodesic with constants depending only on $\rho, w, $ and $( u, v)$.
\end{lemma}
\begin{proof} 
Suppose that $||w||_{(u,v)} = k$ and number the vertices of $\br_{\rho}(w; (u,v))$ in order as $g_rO, r \in \ZZ$. We have to show that there exist constants $K,\e>0$ so that  if  $n<m$ then
$$(m-n )/K - \e\leq  d(g_nO,g_mO) \leq K(m-n) + \e.$$

Pick $c>0$ so  that $  d(O,hO) \leq c $ for  $h \in \{ u,v\}$. Clearly  $d(g_nO,g_mO) \leq c (m-n)$.
For the lower bound, write  $m -n = rk + k_1$ for $r \geq 0, 0 \leq k_1 < k$.
Then  for some cyclic permutation of $w$, say $w'$, we have, setting  $W' = \rho(w')$, 
$W'^r(g_nO) = g_{n+rk}(O)$ so that  $d(g_nO, g_{n+rk}O) \geq r \ell(W)$. 
Thus $$    d(g_nO,g_mO) \geq  d(g_nO,g_{n+rk}O) - d(g_{n+rk}O, g_mO) \geq   (m-n) \ell(W)/k -       kc - \ell(W)/k.$$
\end{proof}

 From now on, we assume that $\rho \in \B$ so that by  Theorem~\ref{thm:BQimpliesBIP},  $\rho$ satisfies $BIP$.   
The following simple consequence of $BIP$  is critical:    
\begin{lemma}\label{projtoaxis} There exists $D>0$ so that for any $u \in \P$ which palindromic with respect to one of the three basic generator pairs, we have $d(U^rO, \Ax U ) <D$ for any $r \in \ZZ$. \end{lemma}
\begin{proof} By $BIP$, we may assume that $\Ax U$ intersects one of the three special hyperelliptic axes at bounded distance at most $D$ to $O$, where $D$ is independent of  $u$. Since $\Ax U$ is invariant under $U$ we have also  $d(U^rO, \Ax U ) <D$ for any $r \in \ZZ$. \end{proof}

\begin{lemma}\label{applyLongbending1}  
Let  $u \in \P$ be palindromic with respect to one of the three basic generators pairs. Let    $\psi$ be the angle at vertex $O$ in the triangle with vertices $O, U^{-1}(O), U(O)$. 
Then  there are $L_1, \e_1 >0$ so that if $\ell(U) > L_1$ then $\psi >\e_1$.
\end{lemma}
\begin{proof} Let $a,b,c$ denote the lengths of the sides $O  U^{-1}(O), O  U(O), U^{-1}(O) U(O)$ respectively. By Lemma~\ref{projtoaxis} 
we have $a \leq \ell(U) +2D, b \leq \ell(U) +2D$  with $D$ as in that lemma.  Clearly $c \geq \ell(U^2) =  2\ell(U)$.
Thus $c \geq a+b -4D$ and the conclusion follows from Lemma~\ref{Longbending}.
\end{proof}

\begin{lemma}\label{applyLongbending}   Suppose that $(u,v)$ is a generator pair palindromic with respect to  one of the three basic generators pairs, and so that $|\Tr UV| \geq |\Tr UV^{-1}|$. Let $\psi$ be the angle at vertex $O$ in the triangle with vertices $O, V^{-1}(O), U(O)$. 
Then  there exist $L_2, \e_2 >0$ so that if $\ell(U), \ell(V) > L_2$ then $\psi >\e_2$.
\end{lemma}
\begin{proof} 
The proof is similar to that of the preceeding lemma with  $a = d(O , V^{-1}(O)), b= d(O,  U(O)),c= d( V^{-1}(O), U(O))$. As before $a \leq \ell(V) +2D, b \leq \ell(U) +2D$ while clearly $c \geq \ell(UV)$.
From Lemma~\ref{compare3} it follows that for large enough $L_2$, 
$$ \ell(UV) \geq \ell(U) + \ell(V) -2 \log 3.$$ Applying Lemma~\ref{Longbending} gives the result.
\end{proof}

Thus we have proved:
\begin{proposition}\label{longwordsqgeod} There exists $L_3>0$ with the following property. Suppose that  for a palindromic generator pair $(u,v)$  we have $\ell(U), \ell(V) >  L_3$ and that $|\Tr UV| \geq | \Tr UV^{-1}|$. 
 Let $\C (u,v)$ denote the set of all cyclically shortest words  in positive powers of $u$ and $v$.
 Then the collection of broken geodesics $\{\br_{\rho}(w; (u,v)), w \in \C  (u,v)\}$ is uniformly quasigeodesic, with constants depending only on $(u,v)$. 
\end{proposition}
\begin{proof} Choose $L_1,L_2$ as in Lemmas~\ref{applyLongbending1} and~\ref{applyLongbending} and then choose $\psi = \min \{\e_1,\e_2\}$. Then choose $L_3 = L(\psi)$ as in Lemma~\ref{brokengeod}.  
\end{proof}

Given a generator pair $(u,v)$, we now
 apply the above results to generator pairs of the form  $(u^Nv, u^{N+1}v),  N \in \ZZ$.

\begin{corollary} \label{applyLongbending2} Let  $(u,v)$ be a generator pair $(u,v)$ such that $(u^{-1}v,v)$ is palindromic and 
let  $N \in \ZZ$. Then there exist 
$m_0 \geq 2,  \e_2>0$, depending on $u,v$ but not on  $N$, with the following property. Suppose that $m \geq m_0$ and that 
$\bu \in \Omega(m)$
while   ${ \bu^N\bv} , {\bu^{N+1}\bv}  \notin \Omega(m)$. Then  the interior angle at $O$ in the triangle with vertices  $(U^NV)^{-1}O , O, U^{N+1}V O$ is at least $\e_2$.
\end{corollary}
\begin{proof} This is Lemma~\ref{applyLongbending} applied to the generator pair $ (u^Nv, u^{N+1}v)$. Choose $L_2$ as in Lemma~\ref{applyLongbending} and then choose $m_0$ so that $|\Tr X | > m_0 $ implies $ \ell(X) > L_0$ for $X \in \SL$ (use $L_0$ as in Lemma~\ref{compare2}).
Note that, for any $m \geq 2$,  since $\Omega(m)$ is connected  and ${ \bu^N\bv} , {\bu^{N+1}\bv}  \notin \Omega(m)$ while $\bu \in \Omega(m)$ it follows that ${ \bu^N\bv \bu^{N+1}\bv} \notin \Omega(m)$ so that $|\Tr U^NV  U^{N+1}V| \geq |\Tr U|$.

Set $x = u^{-1}v$ so that by the hypothesis $(x,v)$ is a palindromic pair. Then, as in the  proof of Proposition~\ref{uniquepalindromes},  $(u^Nv, u^{N+1}v) = ((vx^{-1})^Nv, (vx^{-1})^{N+1}v)$ is also palindromic and the result follows with $\e_2$ as in Lemma~\ref{applyLongbending}. 
\end{proof}

\begin{lemma} \label{replaceUNV} For any palindromic generator pair $(u,v)$, there exist  $n_0 \in \NN$ and $\a>0$, depending only on $U$ and $V$,  so that  $d(O,U^NVO)   \geq \a    |N |$ whenever $ |N |> n_0$.  \end{lemma}
\begin{proof} 
With suitable choice of $n_0$ and $\a$ we have 
\begin{equation}
d(O, U^NVO) \geq d(O, U^NO) -d(O, VO) \geq  |N | \ell(U) - \ell(V) -2D \geq \a  |N | \ \ \mbox{for} \ \  |N| \geq n_0.
\end{equation}
\end{proof} 

%\begin{figure}\label {fig:replaceUNV}
%\includegraphics{}
%\caption{Replacement of $U^NV$}
%\end{figure} 

Our next result allows us to deal simultaneously with all words in  generator pairs $(u^Nv, u^{N+1}v)$, $ N \in \ZZ$.
 \begin{proposition}\label{longwordsqgeod1} Let $(u,v)$ be a  generator pair such that the pair $(u^{-1}v,v)$ is palindromic. Suppose that $\hat w = \hat w(u^Nv, u^{N+1}v)$ for $ N \in \ZZ$ is a word written in positive powers of the generators  $(u^Nv, u^{N+1}v)$ and let $ w$ be the same element of $F_2$ written in terms of $(u,v)$. Suppose that $u,v,m$ satisfy the conditions of Corollary~\ref{applyLongbending2}. Then  there exists $n_1 \in \NN$ so that $ \br_{\rho}(w; (u,v))$ is  quasigeodesic with constants depending on $(u,v)$ but not on $N$ for any $|N| \geq n_1$.  
 \end{proposition}
\begin{proof} 
We have to show that  the distance between \emph{any} two vertices of $ \br_{\rho}(w; (u,v))$ is bounded below by their distance in the word $w$. For simplicity we may assume that $N>0$; the other case is similar.
We first deal with subsegments of the form $x_1x_2 \ldots x_{M} $ where each  $x_i$ is either $ u^Nv$ or $u^{N+1}v$, and handle initial and final segments later. For simplicity we write $\hat N$ to indicate either $N$ or $N+1$ as the case may be.

Label the vertices of  the broken geodesic $ \br_{\rho}(x_1x_2 \ldots x_M; (u^Nv, u^{N+1}v))$ as $P_1, \ldots P_{M+1}$ so that, after translation if needed, $P_1=O, P_2 = X_1O,  P_3 = X_1X_2O,  \ldots, P_{M+1} = X_1X_2 \ldots X_M O$. 

Choose $\e_2>0$ as in Corollary~\ref{applyLongbending2}. Since $|\Tr U^NV| \to \infty $ as $|N| \to \infty$ we can choose $n_1$ so that $|N| > n_1$ implies $\ell(U^NV) > L$ where $L = L(\e_2)$ is as in Lemma~\ref{brokengeod}.
Then  by Lemma~\ref{brokengeod}, the broken geodesic consisting of geodesic segments joining $P_0, P_1, \ldots, , P_{M+1}$   is quasigeodesic, in particular, there exists $c>0$ so that 
\begin{equation} \label{fullperiod} |P_1 P_{M+1}| \geq c(\sum_1^{M} |P_i P_{i+1}|) - c. \end{equation}  
By Lemma~\ref{replaceUNV}  there exists $\a>0$ so that  $  |P_i P_{i+1}| \geq  \alpha \hat N $  giving  (after slight adjustment of constants since $|| x_i||_{(u,v)} = \hat N + 1$  not $\hat N$)     an inequality of the form 
\begin{equation} \label{firsttry} |d(O, X_1X_2 \ldots X_MO)| \geq c'||x_1x_2 \ldots x_M||_{(u,v)} - c. \end{equation}

It remains to deal with possible initial and final segments of $w$ which are not full periods of $u^{\hat N}v$.  This means we have to consider  words of the form $yx_1x_2 \ldots x_{M}z$ where $x_i$ are as above and $y$ and $z$ are respectively initial  and final segments of the form 
$u^hv$ and $u^{h'}$  for $ h,h' \in \{0,\ldots, \hat N-1\}$.

Consider a word of type $yx_1x_2 \ldots x_{M}$; the other cases are similar. If $h \leq r_0$ (for some $r_0$ to be chosen below) then the total length of the broken geodesic $\br_{\rho}(y; (u,v))$ is at most 
$r_0 d(O,UO) + d(O,VO) \leq a(r_0   \ell(U) +   \ell(V) ) < K$ say. Concatenation with such segments at the beginning of $\br_{\rho}(x_1x_2 \ldots x_{M}; (u,v))$
adds at most a bounded constant and the desired inequality follows from~\eqref{firsttry}.

Otherwise adjoin before the initial segment $\br_{\rho}(y; (u,v))$  extra images of $O$ under $U$ so that, again after translation, we have a full period $O, UO,  \ldots, U^{N-h}O, U^{N-h+1}O, \ldots, U^{N}O, U^{N}VO$, with $e$ corresponding to the segment $ U^{N-h+1}O, \ldots, U^{N}O, U^{N}VO$. %, see Figure~\ref{fig:endsegments}. 
Let $P_0 = O, Q = U^{N-h}O, P_1 = U^NVO$ and let $P_2, P_3, \ldots, P_{M+1}$ be the remaining vertices of  $ \br(\hat w; (u^Nv, u^{N+1}v))$, so that $P_1, P_2, \ldots, P_{M+1}$  is the  same broken geodesic as before, translated by $U^NV$.  We need a lower bound of the form $d(Q, P_{M+1}) \geq c (h + ||x_1x_2 \ldots x_M||_{(u,v)} ) - c$.

%\begin{figure}\label {fig:endsegments}
%\includegraphics{}
%\caption{Dealing with the end segments}
%\end{figure} 
 
As above, the broken geodesic joining points $P_0P_1P_2 \ldots P_{M+1}$ is uniformly quasigeodesic (with constants depending on $U,V$ but not $M,N$).   
Let $Q'$ be the foot of the perpendicular from $Q$ onto $P_0P_1$.   Lemma~\ref{replaceUNV} implies that the broken arc joining $P_0= O, UO, \ldots U^NO, U^NVO=P_1$ is quasigeodesic with constants depending only on $(U,V)$. Hence it is within uniformly bounded Hausdorff distance of the arc $P_0P_1$, (\cite{BH} III H Theorem 1.7) so that $|QQ'| < D_0$ for some fixed $D_0$.
Thus we have  comparisons 
\begin{equation} \label{compareperp} |Q'P_i | -D_0  \leq |QP_i| \leq  |Q'P_i | + D_0, i = 1 , \ldots,  M+1.\end{equation} 
By Lemma~\ref{replaceUNV} again we have
$|QP_1| > \a h$, so that 
for  $h>r_0$ for suitable $r_0$ we have  $|Q'P_1| >L$ with $L = L(\e_2)$ as in Lemma~\ref{brokengeod}. Hence   the broken arc  $Q'P_1P_2 \ldots P_{M+1}$  is uniformly quasigeodesic and so $$d(Q,  P_{M+1})  \geq d(Q',  P_{M+1}) -D_0 \geq c(|Q'P_1| + \sum_1^{M} | P_i P_{i+1}|) - c'$$
for suitable $c, c'>0$. Using Lemma~\ref{replaceUNV} once more together with~\eqref{fullperiod} and~\eqref{compareperp} gives the result.
 \end{proof}

We are finally ready to  prove our main result: 
 
  \begin{theorem} \label{prop:BIBQimpliesPS} If  a representation $\rho \co F_2 \to \SL$ satisfies the $BQ$-condition, then $ \rho$ is primitive stable. 
 \end{theorem}
\begin{proof} Suppose that $\rho \in \B$ and let $m\geq 2$. For each $\bu \in \Omega(m)$ choose $ u \in \bu$ and fix some neighbour $\hat u$ of $u$, chosen so that $(u^{-1}\hat u, \hat u)$ is a palindromic pair. Note that the regions adjacent to $\bu$ are of the form $\bu^r \hat \bu$ for $r \in \ZZ$. Let  $\V$ denote the vertices of the tree $\T$.
 For $\q \in \V$, denote by $\N(\q)$ the three regions   abutting at $\q$.
We make the following definitions:
 
 \begin{itemize}
 
  \item  $\Int^{\V}\Omega(m)$ is the set of $\q \in \V$ for which  $|\phi_{\rho}(\bu)| \leq m$ for all $\bu \in \N(\q)$.  
 
 \item  $\dd^{\V}\Omega(m)$ is the set of $\q \in \V$ for which  $|\phi_{\rho}(\bu)| \leq m$ for some $\bu \in \N(\q)$ and 
  $|\phi_{\rho}(\bu')| > m$  for some $\bu' \in N(\q)$.
   
 \item  $\dd_*^{\V}\Omega(m)$ is the set of $\q \in \dd^{\V}\Omega(m)$ for which only one  $\bu \in \N(\q)$ is in $\Omega(m)$.

\item For $N _0 \in \NN$,  $\dd_{N_0}^{\V}\Omega(m)$ 
is the set of $\q \in \dd^{\V}\Omega(m)$ for which $\N(\q) = \{\bu, \bu^N \hat \bu, \bu^{N+1} \hat \bu  \}$ with $\bu \in \Omega(m)$ and  $|N| \geq N_0$.
 \end{itemize}

Choose $L_3$ as in Proposition~\ref{longwordsqgeod} and choose $m>0$ so that $|\phi (X)| > m$  implies $\ell(X) > L_3, X \in \SL$. There are only finitely many regions $\bu$ with $|\phi (\bu)| \leq m$ so $|Int^{\V}\Omega(m)|$ is finite. By Lemma~\ref{finiteboundary},  traces
 increase round $\dd \bu$ for any $\bu \in \Omega(m)$, hence  $\dd^{\V}\Omega(m) \setminus \dd_*^{\V}\Omega(m)  $ is finite.
Finally, choose $N_0$ large  enough for Lemma~\ref{replaceUNV}  and so that 
$\ell(U^N \hat U)> L_3$ whenever $|N| > N_0$.
 It follows that  $  \dd_{N_0}^{\V}(\Omega(m)) \subset   \dd_{*}^{\V}\Omega(m)$   and  that
$\dd^{\V}\Omega(m) \setminus   \dd_{N_0}^{\V}\Omega(m)$ is finite.

\begin{lemma}\label{descendingpath} Suppose that $ \q \notin \Int^{\V} \Omega(m) \cup \dd^{\V} \Omega(m) $. Then $\q$ is connected by a finite descending path to  a vertex in $\dd^{\V} \Omega(m) $. Moreover the first such point along this path is   in  $\dd_*^{\V} \Omega(m) $.\end{lemma}
\begin{proof} By Lemma~\ref{infiniteray} there is a finite descending path to an edge one of whose neighbouring regions is in $\Omega(2)$.  So there must be a first vertex $\q_*$ exactly  one of whose neighbours is in $\Omega(m) $, so that $\q_* \in \dd_{*}^{\V} \Omega(m) $. \end{proof}

Given a vertex $\q \in \dd_{*}^{\V} \Omega(m) $,  there is a unique oriented edge $\vec e$  whose head is $\q$  and which is not contained in $ \dd^{\V} \Omega(m) $.
We define  the  \emph{wake} of $\q$ to be the wake of $\vec e$.
 
 Suppose that a region $\bv$ is not in $\Omega(m)$. If $\bv$ has a vertex in common with $\Omega(m)$, then it has an edge $e$ say in common with $\Omega(m)$. Let $\bz \in \Omega(m)$ be the other region adjacent to $e$, and let $\by, \by'$ denote the  other regions adjacent to $\bv$ at the two ends of $e$. Note there are at most finitely many regions $\bv$ for which at least one of $\by, \by'$ is in $\Omega(m)$; for all other $\bv$ the two end vertices of $e$ are in $\dd_{*}^{\V}\Omega(m)$.

Otherwise $\bv$ has no vertex in common with $\Omega(m)$. Pick any oriented edge $\vec e$ of  $\dd \bv$.  It follows from Lemma~\ref{descendingpath} that $\vec e$ is connected   by a finite descending path which first meets $\dd^{\V} \Omega(m) $ in some vertex $ \q_{*} (\bv) \in \dd_*^{\V} \Omega(m) $, and $\bv$  is in the wake of  $\q_{*}(\bv)$.
 Moreover since $\dd_{*}^{\V}\Omega(m) \setminus  \dd_{N_0}^{\V}\Omega(m)$ is finite, all but finitely many of these wakes 
 land in  $  \dd_{N_0}^{\V}(\Omega(m)$.

In summary we have shown that any region  is either in $\Omega(m)$; or has an  edge in common with $\Omega(m)$ with at least one other adjacent region in $\Omega(m)$; or is in the wake of a descending path which lands at  some point $ \q_{*} \in \dd_*^{\V} \Omega(m) \setminus  \dd_{N_0}^{\V}\Omega(m$; or finally is in the wake of a descending path which lands at  $\q_{*} \in \dd_{N_0}^{\V} \Omega(m) $.

By Lemma~\ref{singleqgeod}, the broken geodesic $\br_{\rho}(w; (a,b))$ constructed from each of the finitely many regions $\bu$ of the first two types is quasigeodesic  with constants depending on $\bu$ for any $ w  \sim u$.

Suppose that $\bz$ in the wake of a vertex in $\dd_{*}^{\V}\Omega(m) \setminus  \dd_{N_0}^{\V}\Omega(m)$. 
Let $(x,y)$ be the palindromic pair of generators adjacent to the edge $e$ whose head is  $\q_{*} (\bz)$. (These will be of the form $(u^{n} \hat u, u^{n+1} \hat u)$ for some $n\in \ZZ$.) Since  $\ell(X), \ell(Y) > L$, since the arrow on $e$ points into $\Omega(m)$, and since $z$ is, up to cyclic permutation,  a product of positive powers of $x,y$, by Proposition~\ref{longwordsqgeod}
the collection of such broken geodesics $\br_{\rho} (z; (x,y))$ is uniformly quasigeodesic with constants depending on $(x,y)$.

Finally suppose that $\bz$ is in the wake of a descending path which lands at  $\q_*(\bz) \in \dd_{N_0}^{\V} \Omega(m) $.  
The neighbours of the edge whose head is $\q_*(\bz)$ are of the form $(\bu^n \hat \bu , \bu^{n+1} \hat \bu )$ where $|n| \geq N_0$. Note that, up to cyclic permutation, any  $z \in \bz$ can be written as a product of positive powers of $(u^n \hat u , u^{n+1} \hat u)$. Hence 
by Proposition~\ref{longwordsqgeod1},  the collection of  such broken geodesics $\br_{\rho}(z; (u, \hat u))$   is uniformly quasigeodesic with constants depending only on $(u,\hat u)$.

Putting all this together,  there is a finite set of generator pairs $\S$, such that any   $w \in F_2$ can be expressed as a word in some $(s,s' ) \in \S$  in such a way that   $\br_{\rho}(w; (s,s' ))$  is quasigeodesic with constants depending only on $(s,s' )$. The quasigeodesic  $\br_{\rho}(w; (s,s' ))$ can be replaced by a broken geodesic $\br_{\rho}(w; (a,b ))$
which is also quasigeodesic  with a change of constants. The total number of  replacements required involves only finitely many constants and the result follows.
 \end{proof}

%%%%%%%%%%%%%%%%%%%%%%%%%%%%%%%%%%%%%%%%%%%%%%% 
 
%%%%%%%%%%%%%%%%%%%%%%%%%%%%%%%%%%%%%%%%%%%%%%%%%%%%%%%%

%%%%%%%%%%%%%%%%%%%%%%%%%%%%%%%%%%%%%%%%%%%%%%%%%%%%%%%%%%%%%%%%%%%%%%

\begin{thebibliography}{000}
%%%%%%%%%%%%%%%%%%%%%%%%%%%%%%%%%%%%%%%%%%%%%%% 
 
 
   \bibitem{BSeries}
J. Birman and C. Series.
\newblock  Geodesics with multiple self-intersections and 
symmetries on Riemann surfaces.
\newblock In {\em Low dimensional topology and Kleinian 
groups},   D. Epstein  ed., LMS Lecture Notes 112, Cambridge Univ. Press, 3 -- 12,   1986.

\newblock {\em Proc. London Math. Soc.  77}, 697--736, 1998.

   
  \bibitem{BH}
M.~Bridson and A. Haefliger.
\newblock {\em Metric spaces of non-positive curvature.}
\newblock Springer Grundlehren Vol. 319, 1999.
  
   \bibitem{bow_mar}
B.~H. Bowditch.
\newblock {M}arkoff triples and quasi-{F}uchsian groups.
\newblock {\em Proc. London Math. Soc.  77}, 697--736, 1998.





   \bibitem{gilmankeen1}
J.~Gilman and L.~Keen.
\newblock Enumerating palindromes and primitives in rank two free groups
\newblock {\em Journal of algebra}, 332, 1--13, 2011.

 \bibitem{gilmankeen2}
J.~Gilman and L.~Keen.
\newblock Discreteness criteria and the hyperbolic geometry of palindromes.
\newblock {\em Conformal geometry and dynamics}, 13, 76 --90, 2009.



\bibitem{goldman}
W. Goldman. 
\newblock The modular group action on real $SL(2)$-characters of a one-holed torus.
\newblock {\em Geometry and Topology}  7, 443 -- 486, 2003.


\bibitem{goldman2}
W. Goldman. 
\newblock Trace coordinates on Fricke spaces of some
simple hyperbolic surfaces.
\newblock In {\em Handbook of Teichm\"uller theory Vol. II},  IRMA Lect. Math. Theor. Phys., 13, Euro. Math. Soc., Z\"urich, 611-- 684, 2009.


 
\bibitem{ksriley}   L. Keen and C. Series. 
\newblock The Riley slice of Schottky space. 
\newblock {\em Proc. London Math. Soc.}, 69,  72 -- 90, 1994.



\bibitem{binbin} J. Lee and B. Xu.
 \newblock Bowditch's Q-conditions and Minsky's primitive stability.
 \newblock {\em arXiv:1812.04237 [math.GT]},  2018.
 

 
 
\bibitem{lupi}   D. Lupi. 
\newblock Primitive stability and Bowditch conditions for rank 2 free group representations. 
\newblock   {\em Thesis, University of Warwick},  2016.



\bibitem{Minsky}   Y. Minsky. 
\newblock On dynamics of $Out(F_n)$ on $PSL(2,\CC)$ characters. 
\newblock {\em Israel Journal of Mathematics}, 193,  47 -- 70, 2013.


 
\bibitem{serwolp}
C.~Series.
\newblock  An extension of Wolpert's derivative formula.
\newblock {\em Pacific J. Math.}, 
197, 223 -- 239, 2001.

  
\bibitem{sty}
 C.~Series, S.P.~Tan, Y.~Yamasita.
\newblock The diagonal slice of Schottky space. 
\newblock {\em Algebraic and Geometric Topology}, 17, 2239 -- 2282,  2017.  



\bibitem{tan_gen}
S.P.  Tan, Y. L.  Wong  and Y. Zhang.
\newblock Generalized {M}arkoff maps and {M}c{S}hane's identity.
\newblock {\em Adv. Math.  217}, 761--813, 2008.




%%%%%%%%%%
\end{thebibliography}
\end{document}